\documentclass[10pt,twoside]{amsart}

\usepackage{amsmath,amssymb
}
\usepackage{amsfonts}
\usepackage{amsthm,amscd}
\usepackage{amsmath}
\usepackage{amssymb}
\usepackage{amstext}
\usepackage{color}
\usepackage{latexsym}
\usepackage{amscd,graphics}

\usepackage{bbm}
\usepackage{empheq}
\usepackage{mathtools}
\PassOptionsToPackage{hyphens}{url}
\usepackage{hyperref}
\hypersetup{colorlinks = false, urlcolor = blue, bookmarksopen = true}

\newcommand{\comments}[1]{} 

\def\Var{\mathop{\rm Var}\nolimits}

\newtheorem{proposition}{Proposition}[section]
\newtheorem{lemma}[proposition]{Lemma}

\newtheorem{theorem}[proposition]{Theorem}

\theoremstyle{remark}
\newtheorem{remark}[proposition]{Remark}
\theoremstyle{definition}

\numberwithin{equation}{section}

\begin{document}

\title[Logarithmic bounds on ergodic sums]{Logarithmic bounds for ergodic sums of\\ certain flows on the torus: a short proof}

\author{J{\'e}r{\^o}me Carrand}
\address{Laboratoire de Probabilit\'es, Statistiques et Mod\'elisation (LPSM),  
CNRS, Sorbonne Universit\'e, Universit\'e de Paris,
4, Place Jussieu, 75005 Paris, France}
\email{jcarrand@lpsm.paris}

\date{\today}
\begin{abstract}
We give a short proof that the ergodic sums of $\mathcal{C}^1$ observables for a $\mathcal{C}^1$ flow on $\mathbbm{T}^2$ admitting a closed transversal curve whose Poincar{\'e} map has constant type rotation number have growth deviating at most logarithmically from a linear one. 
For this, we relate the latter integral to the Birkhoff sum of a well-chosen observable on the circle and use the Denjoy-Koksma inequality. 
We also give an example of a nonminimal flow satisfying the above assumptions.
\end{abstract}
\thanks{I thank S. Ghazouani for allowing me to use his idea for the proof and Y. Coud{\`e}ne for many useful comments. Research supported by the European Research Council (ERC) under the European Union's Horizon 2020 research and innovation programme (grant agreement No 787304).}
\maketitle

\section{Introduction}

Since the work of Furstenberg \cite{furstenberg1973ue}, it is known that the classical horocycle flow of a compact surface of constant negative curvature is uniquely ergodic --- it has only one invariant Borel probability measure. This flow is related to a hyperbolic one, namely the geodesic flow, in the sense that the horocycle orbits are the unstable manifolds for the geodesic flow.

Using Symbolic Dynamics arguments (resp. equicontinuity of some functions), Marcus \cite{marcus1975axiomadiffeo} (resp. \cite{marcus1975axiomaflow}) generalized this result to the flow generated by the orientable one-dimensional unstable foliation of a connected basic piece of an Axiom A diffeomorphism (resp. flow). Later, Bowen and Marcus \cite{bowen1977ue} extended this result to the higher dimensional strong stable or strong unstable foliation of a basic set for an Axiom A diffeomorphism or flow.

In their pioneer work, Giulietti and Liverani \cite{GL2019} focused on the one-dimensional stable foliation of a $\mathcal{C}^r$ Anosov diffeomorphism $F$ of the two-torus, inducing a flow $h^t$ called the Giulietti--Liverani (stable horocycle) flow (of $F$). 
Giulietti and Liverani proved that this flow is uniquely ergodic, minimal and that it admits a closed transverse curve such that the rotation number of the first return map to this curve is of constant type. For more basic facts about this flow, see \cite[Appendix A]{baladi2019there}.

For any continuous function $f:\mathbbm{T}^2 \to \mathbbm{C}$, any $T>0$ and any $x \in \mathbbm{T}^2$, define the horocycle integral $H_{x,T}(f) = \int_{0}^T f(h^t(x)) \, \mathrm{d}t$. 
By unique ergodicity, we have for any such $x$ and $f$, \[\lim\limits_{T \to \infty} \frac{H_{x,T}(f)}{T} = \mu^s(f) \coloneqq \int_{\mathbbm{T}^2} f \, \mathrm{d}\mu^s,\] where $\mu^s$ is the unique invariant probability measure of the flow $h^t$.

For large enough $r$, Giulietti and Liverani introduce a transfer operator for $F$ on some suitable Banach space. Using eigenvectors of the dual operator associated to eigenvalues with modulus larger than the essential spectral radius (Ruelle resonances), they give an asymptotic expansion of $H_{x,T}(f)$ \cite[Theorem 2.8]{GL2019}. The dominant term is the term $T \mu^s(f)$, corresponding to the trivial resonance $\lambda_0 = e^{h_{top}}$, where $h_{top}$ is the topological entropy of $F$. This expansion also involves a negative power law error term. A simpler asymptotic expansion, in the case where all Ruelle resonances of the transfer operator have trivial Jordan blocks, can be found in \cite[Equation (1.2)]{baladi2019there}.

In their recent works, V.~Baladi \cite{baladi2019there} and G.~Forni \cite{forni2020equidistribution} independently proved that horocycle integrals (in the set-up from \cite{GL2019}) do not have deviations, in other words the expansion is limited to the linear term with a bounded remainder. Their proofs are quite different: V.~Baladi proves the strong result that the map $F$ does not have non-trivial Ruelle resonance, while G.~Forni uses the action of the (pseudo-)Anosov diffeomorphism on the first cohomology --- in the more general setting of surfaces of genus $g \geqslant 1$ (non-trivial Ruelle resonances can appear only for $g \geqslant 2$).

\vspace{1em}

In this note we give a new, much shorter, proof of the absence of deviations for horocycle integrals by considering a slightly more general setting: we no longer assume that the flow can be obtained from the stable foliation of an Anosov diffeomorphism. Instead, we only assume that the flow can be recovered from the suspension of a circle diffeomorphism whose rotation number is of constant type. In particular, these flows are uniquely ergodic. For clarity, we  call ``ergodic integral'' for this type of flows the quantity defined as ``horocycle integral'' previously.

We give an elementary proof that the ergodic integral of a $\mathcal{C}^1$ observable along the trajectory of such a flow on the two-torus grows at most logarithmically if the observable has zero average with respect to the unique invariant measure of the flow. This is the content of our main theorem (Theorem \ref{thm:_main}).

When comparing this estimate to the asymptotic expansion given by Giulietti and Liverani \cite[Theorem 2.8]{GL2019}, this result gives a new proof of the absence of deviations for the horocycle integral.

Finally, we prove that the class of flows we consider here is strictly larger than the class of flows studied by Giulietti and Liverani by constructing a flow satisfying our assumptions but which is not minimal --- in contrast to all flows in \cite{GL2019}. This is the content of Theorem \ref{thm:_nonminimal_flow}.

\section{Main result}

Given a flow $h_t$ on the two-torus, we call \emph{ergodic integral} of an observable $f : \mathbbm{T}^2 \to \mathbbm{R}$ at $x \in \mathbbm{T}^2$ and $T >0$ the quantity $H_{x,T}(f) \coloneqq \int_{0}^T f \circ h_t(x) \, \mathrm{d}t$.

Recall the following classical theorem --- we give a short proof of this fact using results from \cite{katok1997introduction} in order to introduce notations for our main result. In particular the theorem below gives a simple sufficient condition for a flow to be written as the suspension of a circle diffeomorphism.

\begin{theorem}\label{thm:_KH}
If $h_t$ is a $\mathcal{C}^1$ flow on the torus $\mathbbm{T}^2$ without critical points nor periodic orbits, then there exists a smooth closed curve $\gamma$ transverse to $h_t$ such that $h_t$ is smoothly conjugated to the suspension of the first return map $R: \gamma \to \gamma$.\\
Moreover, the flow $h_t$ is uniquely ergodic, with a unique invariant measure $\mu$.
\end{theorem}

Recall that an irrational number is of constant type if the sequence $(a_k)_k$ of its coefficients in its continued fraction expansion is bounded. We can now state our main result, using notations from the previous theorem.

\begin{theorem}\label{thm:_main}
If $h_t$ is a $\mathcal{C}^1$ flow on the torus $\mathbbm{T}^2$ without critical point nor periodic orbit, and if the rotation number of the Poincar{\'e} first return map $R$ is of constant type, then there exist constants $K_1$ and $K_2$ such that for any $\mathcal{C}^1$ observable $f$ with $\int f \, \mathrm{d}\mu =0$, any $x$ and any $T>0$, \[ |H_{x,T}(f)| \leqslant K_1 ||f||_{\mathcal{C}^1} \log(1+ T) + K_2 ||f||_{\mathcal{C}^1}.\]
\end{theorem}

More precise versions of that estimate in the case of Giulietti--Liverani flows can be found in \cite{baladi2019there} and in \cite{forni2020equidistribution}. The bound obtained by V.Baladi \cite{baladi2019there} is much tighter --- but the proof is longer --- while the estimate given by G.Forni \cite{forni2020equidistribution} applies to flows on higher genus surfaces.

\begin{proof}[Proof of Theorem~\ref{thm:_KH}] By the Birkhoff recurrence theorem, any continuous transformation of a compact space has a recurrent point. Hence $h_1$ has recurrent orbits. In particular the flow $h_t$ also has recurrent points. By our assumptions on the flow, these orbits cannot be periodic. Hence, by \cite[Propositions 14.2.1 and 14.2.3]{katok1997introduction} there exists a smooth closed curve $\gamma$ \emph{transverse} to $h_t$ and parametrised by $\mathbbm{S}^1$ such that every orbit of $h_t$ intersects $\gamma$. We can therefore apply \cite[Corollary 14.2.3]{katok1997introduction} to get that $h_t$ is smoothly conjugated to the suspension flow  of the first return map $R$ to $\gamma$. 
The conjugation is $\mathcal{C}^1$, since the change of coordinates is $(\theta,t) \mapsto h_t (\theta)$.

The map $R: \mathbbm{S}^1 \to \mathbbm{S}^1$ is a $\mathcal{C}^1$ diffeomorphism of the circle which has no periodic point. It is a classical result --- see \cite[Theorem 3.3.5]{cornfeld2012ergodic} --- that $R$ is uniquely ergodic, with invariant measure $\nu$, and that its rotation number is irrational. From this, we deduce that $h_t$ is uniquely ergodic, with a unique invariant measure $\mu$.
\end{proof}

We can now give the proof of our main result.

\begin{proof}[Proof of Theorem~\ref{thm:_main}]
Suppose that the rotation number $\omega$ of $R$ is of constant type. In order to prove the estimate, we will compare the ergodic integral to the Birkhoff sum of an appropriate function. 

Let $u:\mathbbm{S}^1 \to \mathbbm{R}_+$ be the first return time function to $\gamma$, and let $f: \mathbbm{T}^2 \to \mathbbm{R}$ be a $\mathcal{C}^1$-observable such that $\int_{\mathbbm{T}^2} f \, \mathrm{d}\mu = 0$. By construction, $\gamma$ is a smooth curve, uniformly transverse to the flow, hence the function $u$ is of class $\mathcal{C}^1$. Define the $\mathcal{C}^1$ observable $g$ on $\gamma$ by the formula \[ g(x) = \int_{0}^{u(x)} f \circ  h_t(x) \, \mathrm{d}t. \]

To estimate the ergodic integral of $f$ by the Birkhoff sum of $g$ under the map $R$, we use the following lemma.

\begin{lemma}\label{link_T_n}
For all $x \in \gamma$ and $T>0$ there exists $n$ satisfying $\frac{T}{\sup (u)} -1 \leqslant n \leqslant \frac{T}{\inf (u)}$ and such that \[ \left| H_{x,T}(f) - \sum\limits_{k=0}^{n-1} g \circ R^k(x) \right| \leqslant \sup(u) \sup |f|.\]
For all $y \in \mathbbm{T}^2$ there is $0 \leqslant \tau < \sup u$ and $x \in \gamma$ such that $y=h_{\tau}(x)$ and 
\[ \left| H_{x,T+\tau}(f) - H_{y,T}(f) \right| \leqslant \sup(u) \sup |f|. \]
\end{lemma}

\begin{proof}
We first determine $n$. Since $\inf u > 0$, there exists $n$ such that $\sum\limits_{k=0}^{n-1} u \circ R^k(x) \leqslant T < \sum\limits_{k=0}^{n} u \circ R^k(x)$. Hence $n \inf u \leqslant T$ and $(n+1) \sup u \geqslant T$. Both estimates on ergodic integrals then follow from the fact that $h_t(R^n(x)) = h_{t + \sum_{k=0}^{n-1} u(R^k(x))}(x)$ for all $x \in \gamma$ and all $t \in \mathbbm{R}$, .
\end{proof}

In order to conclude by applying the Denjoy--Koksma theorem \cite[Theorem VI.3.1]{herman1979conjugaison}, we also need the following lemma.

\begin{lemma}\label{decomposition_n}
If $\omega = [0,a_1, \ldots, a_k,\ldots]$ is of constant type, then for any integer $n>1$ there exists integers $N$ and $(n_1,\ldots, n_N)$ such that $ n-1 = \sum\limits_{k=0}^{N} n_k q_k$, where $\tfrac{p_k}{q_k} = [0,a_1, \ldots, a_k]$. \\
Furthermore, we can choose $N < 4\log(n)/\log(2)$ and $n_k \leqslant B$ for all $k$, where $B$ is a bound on the coefficients $(a_k)_{k\geqslant 1}$. 
\end{lemma}

\begin{proof}
Since the sequence $(q_k)_{k \geqslant 0}$ satisfies the recursion formula $q_{k+1} = a_k q_k + q_{k-1}$ with $q_0=1$ and $q_1=a_1$, we get by induction that $2 ^{\frac{k-1}{2}} \leqslant q_k $. Therefore, there exists $N$ such that $q_{N} \leqslant n -1 < q_{N+1}$ with the estimate $N < 4\log(n)/\log(2)$.

Define inductively the sequences $(r_k)_{0 \leqslant k \leqslant N+1}$ and $(n_k)_{0 \leqslant k \leqslant N}$ by $r_{N+1} \coloneqq n-1$ and the Euclidean division $r_{k+1} = n_k q_k + r_k$, with $0 \leqslant r_k < q_k$. Clearly, we get that $n-1 = \sum\limits_{k=0}^{N} n_k q_k$ (because $q_0=1$). By contradiction, suppose there exists $k$ such that $n_k>B+1$. Then \[ r_{k+1} = n_k q_k + r_k > (B+1)q_k +r_k > a_{k+1}q_k + q_{k-1} +r_k = q_{k+1} + r_k .\]
Therefore $r_{k+1} \geqslant q_{k+1}$, which is a contradiction. Hence $n_k \leqslant B$ for all $k$.
\end{proof}

For completeness, we state the Denjoy--Koksma inequality:
\begin{theorem}[Denjoy--Koksma inequality]
Let $f$ be a homeomorphism of the circle with an irrational rotation number $\rho(f)$. Let $\mu$ be a measure invariant by $f$, and let $p/q$ be such that ${\rm gcd}(p,q)=1$ and $| q\rho(f) - p | < 1/q$. Then for all potential $\varphi$ of bounded variation and all $x \in \mathbbm{S}^1$, $\left| \sum\limits_{k=0}^{q-1} \varphi \circ f^k(x) - q \int \varphi \, \mathrm{d}\mu \right| < \mathrm{Var}(\varphi)$.
\end{theorem}

Since $g$ is $\mathcal{C}^1$, it is of bounded variation. In addition, the denominators $(q_k)_{k \geqslant 0}$ associated to $\omega$ satisfy the assumption $|q_k \omega -p_k | < 1/q_k$ for some integer $p_k$ coprime with $q_k$. We can therefore apply the Denjoy--Koksma theorem to $g$, $R$ and any $q_k$. Furthermore notice that, by construction, $g$ is of $\nu$-average $0$: indeed, let $M=\{(x,t) \mid x \in \gamma, \, t\in [0,u(x)] \}/\sim$, with $(x,u(x))\sim (R(x),0)$, be the space such that $h_t$ is conjugated with its unit speed vertical flow. Let $\bar{\mu}$ be the image of $\mu$ by the conjugacy map. Thus, $\bar{\mu}$ is invariant by the vertical flow and so it must be of the form $\bar{\mu} = \frac{1}{\int u \,\mathrm{d}\bar{\nu}} \bar{\nu} \otimes \mathrm{d}t$, where $\bar{\nu}$ is invariant under $R$. By unique ergodicity of $R$, we have $\bar{\nu}=\nu$. Thus
\begin{align*}
0 &= \int_{\mathbbm{T}^2} f \,\mathrm{d}\mu = \int_{M} f(h_t(x)) \,\mathrm{d}\bar \mu(x,t) \\
&= \frac{1}{\int u \,\mathrm{d}\nu} \int_\gamma \int_0^{u(x)} f(h_t(x)) \,\mathrm{d}t \,\mathrm{d}\nu(x) = \frac{1}{\int u \,\mathrm{d}\nu} \int g \,\mathrm{d}\nu \, .
\end{align*} 

Fix $x \in \mathbbm{T}^2$ and $T > 0$. By Lemma \ref{link_T_n}, there exist a point $y \in \gamma$ and an integer $n$ from which we can estimate the ergodic integral of $f$ at $x$ and $T$ with the Birkhoff sum of $R$ at $y$. In order to assume that $n > 1$, we assume that $T > 2 \sup u$ (otherwise, the theorem holds with $K_1=0$ and some $K_2>0$ depending only on $u$). By Lemma \ref{decomposition_n} we can decompose $n-1$ as a sum from which we deduce the equality
\[ \sum\limits_{k=0}^{n-1} g \circ R^k(y) = \sum\limits_{l=0}^{N} \sum\limits_{m=0}^{n_l-1} \sum\limits_{k=0}^{q_l -1} g \circ R^k \left( R^{m q_l + \sum\limits_{i=0}^{l-1} n_i q_i} y \right). \]
From the Denjoy-Koksma inequality, for all $0 \leqslant l \leqslant N$, all $0 \leqslant m < n_l$ and all $y$ in $\gamma$,
\[ \left| \sum\limits_{k=0}^{q_l -1} g \circ R^k \left( R^{m q_l + \sum\limits_{i=0}^{l-1} n_i q_i} y \right) \right| < \Var(g), \]
we deduce the estimate
\[ \left| \sum\limits_{k=0}^{n-1} g \circ R^k(y) \right| \leqslant N B \Var(g) \leqslant \frac{4B \Var(g)}{\log 2} \log n  \leqslant \frac{4B \Var(g)}{\log 2} \log \frac{T}{\inf(u)} .\]
Hence the result, 
\begin{align*}
\left|H_{x,T} (f) \right| &\leqslant \left| H_{x,T}(f) - H_{y,T-\tau}(f) \right| + \left| H_{y,T-\tau}(f) - \sum\limits_{k=0}^{n-1} g\circ R^k(y) \right| + \left| \sum\limits_{k=0}^{n-1} g\circ R^k(y) \right|,\\
&\leqslant \frac{4B \Var(g)}{\log 2} \log \frac{T}{\inf(u)} + 2 \sup (u) \sup|f| \eqqcolon \tilde{K}_1 \log T + \tilde{K}_2.
\end{align*} 
We can bound the total variation $\Var(g)$ by the product of the length of $\gamma$ with $||g'||_{\mathcal{C}^0(\gamma)}$. By the definition of $g$, we get \[ ||g'||_{\mathcal{C}^0(\gamma)} \leqslant ||u'||_{\mathcal{C}^0(\gamma)} \, ||f||_{\mathcal{C}^0} + ||u||_{\mathcal{C}^0(\gamma)} ||\mathrm{d}f||_{\mathcal{C}^0} \sup\limits_{0 \leqslant t \leqslant ||u||_{\mathcal{C}^0(\gamma)}} ||\mathrm{d}h_t||_{\mathcal{C}^0}. \] Notice that $||u'||_{\mathcal{C}^0(\gamma)}$ and $\sup\limits_{0 \leqslant t \leqslant ||u||_{\mathcal{C}^0(\gamma}} ||\mathrm{d}h_t||_{\mathcal{C}^0}$ only depend on the flow $h_t$ and on $\gamma$. Hence there exist constants $K_1$ and $K_2$ that depend only on $h_t$ such that $\tilde{K}_1 \leqslant K_1 ||f||_{\mathcal{C}^1}$ and $\tilde{K}_2 \leqslant K_2 ||f||_{\mathcal{C}^1}$.
\end{proof}

Finally, remark that in order to get a rotation number of constant type, the condition for the flow not to have periodic orbit is necessary: otherwise the existence of a transverse curve $\gamma$ is no longer guaranteed. If such a curve exists then the first return map $R$ has a periodic point, hence has a rational rotation number.

\section{A nonminimal flow satisfying the assumptions of Theorem \ref{thm:_main}}

We finish this note by proving that the class of flows we are working with is strictly larger than the class of flows studied by Giulietti and Liverani which are necessarily minimal. The proof relies on constructing a family of $\mathcal{C}^1$ nonminimal flows. By \cite[Proposition~14.2.4]{katok1997introduction}, these flows are less than $\mathcal{C}^2$.

\begin{theorem}\label{thm:_nonminimal_flow}
There exists a flow on $\mathbbm{T}^2$ satisfying the assumptions of Theorem~\ref{thm:_main} that is not minimal. Furthermore, the flow can be chosen to be renormalized by an Axiom A diffeomorphism.
\end{theorem}

Notice however that all flows satisfying the assumptions of Theorem~\ref{thm:_main} are obtained by suspending circle diffeomorphisms of irrational rotation numbers, and thus are minimal on the support of their unique invariant measure.

Without the last condition of renormalization, we can simply construct such a flow by taking the suspension of a Denjoy counter-example whose rotation number is of constant type. Such circle diffeomorphisms exist by the original construction of Denjoy, which works for any irrational rotation number. For an expository on the construction of Denjoy counter-examples, see for example\footnote{I thank Selim Ghazouani for indicating me this reference.} \cite{athanassopoulos2015denjoy}. However, there is no reason for the flow obtained by suspending a Denjoy counter-example to be renormalized by an Axiom A diffeomorphism. Adding this condition, the flow falls into the category of $W^u$-flows studied by Marcus in \cite{marcus1975axiomadiffeo}, in the particular case where the phase space of the flow is the same as the one of the Axiom A map --- in opposition with just the set of nonwandering points of the map. Finally, results on Ruelle spectrum and dynamical determinants for Axiom A diffeomorphisms can be found in \cite{baladi2008dyndet, dang2021} (and results on dynamical zeta functions for Axiom A flows in \cite{dyatlov2018dynzeta}), but asymptotic expansions of ergodic integrals associated to $W^u$-flows using transfer operator techniques are still quite rare in literature and there is room for work to be done in this setting.

In order to build a flow satisfying this last condition, consider the derived from Anosov transformation on the two-torus studied in \cite[Chapter 9]{coudene2016attractor} and \cite{coudene2006pictures}. Recall some notation. Starting from Arnold's \emph{cat map} (case $\beta =0)$ in the diagonalized form, and adding a bump in the unstable direction, let $f_{\beta} : \left[ -\tfrac{1}{2}, \tfrac{1}{2} \right]^2 \to \mathbbm{R}^2$ be as follows \[ f_{\beta} \begin{pmatrix} x \\ y \end{pmatrix} \coloneqq \frac{1}{1+\lambda^2} \begin{pmatrix} \lambda & -1 \\ 1 & \lambda \end{pmatrix} \begin{pmatrix}
\lambda^2 + \beta k\left( \frac{\sqrt{x^2+y^2}}{2} \right) & 0 \\ 0 & \lambda^{-2} \end{pmatrix} \begin{pmatrix} \lambda & 1 \\ -1 & \lambda \end{pmatrix} \begin{pmatrix} x \\ y \end{pmatrix}, \]
where $\lambda = \frac{1+\sqrt{5}}{2}$, $ - \lambda^2 < \beta < 0 $ and $k$ is an even, unimodal function supported in $[-1,1]$ such that $k(0)=1$ -- \emph{e.g.} $k(r)=(1-r^2)^2 \mathbbm{1}_{[-1,1]}(r)$ -- so that the map $f_{\beta}$ is invariant by the action of $\mathbbm{Z}^2$ and induces a map, also called $f_{\beta}$, on the torus $\mathbbm{T}^2$. It is shown in \cite[Chapter 9]{coudene2016attractor} that $f_{\beta}$ is a diffeomorphism of class $C^1$ of the torus and if $ - \lambda^2 < \beta < - \lambda^2 +1 $ then the origin is an attractive hyperbolic fixed point. Let $K_\beta$ be the invariant subset defined as the complement of the basin of attraction of $0$. This map is an explicit example of Smale's derived from Anosov transformation as introduced in \cite[Section I.9]{smale1967differentiable}, here obtained by perturbing Arnold's \emph{cat map}.

Let $e_u = \frac{1}{\sqrt{1 + \lambda^2}} \begin{pmatrix}
\lambda \\ 1
\end{pmatrix}$ and $e_s = \frac{1}{\sqrt{1 + \lambda^2}} \begin{pmatrix}
-1 \\ \lambda
\end{pmatrix}$ be unitary eigenvectors of the matrix $A \coloneqq \begin{pmatrix}
2 & 1 \\ 1 & 1
\end{pmatrix}$ respectively associated to eigenvalues $\lambda^2$ and $\lambda^{-2}$. Since $A$ is symmetric, notice that $(e_u,e_s)$ is an orthonormal basis. In this basis the Jacobian matrix of $f_{\beta}$ is
\begin{align*}
\mathrm{Jac}(f_{\beta})(x) = \begin{pmatrix}
a_{\beta}(x) & b_{\beta}(x) \\ 0 & \lambda^{-2}
\end{pmatrix}.
\end{align*}
Since the Jacobian is upper-triangular, lines spanned by $e_u$ are stable by $f_\beta$. Assuming that $k$ satisfies also $k + id \, k' \leqslant 1$, $f_\beta |_{K_\beta}$ expands uniformly the direction spanned by $e_u$. In order to construct a stable foliation over $K_\beta$, for $X$ a vector field, denote $(f_\beta)_*X(x) = (\mathrm{d}_x f_\beta)^{-1} X(f(x))$ to be the pullback of $X$ by $f_\beta$. Formally, if $v^s_{\beta} = \lim\limits_{n \to + \infty} \lambda^{-2n} (f_\beta)^n_* X$, then $\lambda^{-2}(f_\beta)_*v^s_\beta = v^s_\beta$, or in other words $\mathrm{d}_x f_\beta \, v^s_\beta(x) = \lambda^{-2} v^s_\beta(f(x))$, $v^s_\beta$ is uniformly contracted by $\mathrm{d}f_\beta$. For the constant vector field $X \equiv e_s$, formally we get
\begin{align}\label{eq:vs}
v_{\beta}^s(x) = e_s - \sum\limits_{i=0}^{\infty} \lambda^{-2i} b_{\beta}(f_\beta^{i}(x)) \prod\limits_{j=0}^i \frac{1}{a_{\beta}(f_\beta^j(x))} \ e_u, \quad x \in \mathbbm{T}^2.
\end{align}
This equation being only formal, we need to check that the series inside it converges. Since $b_\beta$ is bounded and $a_\beta > 1$ on the compact set $K_\beta$, \eqref{eq:vs} defines a vector field on $K_\beta$, uniformly contracted by $f_\beta$:
\begin{align}\label{eq:uniform_contraction_vs}
\mathrm{d}_x f_\beta \, v^s_\beta(x) = \lambda^{-2} v^s_\beta(f_\beta(x))
\end{align}
for all $x \in K_\beta$. It is shown in \cite[Theorems 3.3 and 3.6]{carrand2020ARPE} --- in a slightly more general context --- that \eqref{eq:vs} defines a Lipschitz continuous vector field on $\mathbbm{T}^2$ for any fixed $\beta$ in $] -\lambda^{2} + \lambda^{-4}, 0]$ and that the map $(x,\beta) \mapsto v_{\beta}^s(x)$ is continuous on $\mathbbm{T}^2 \times \, ] -\lambda^{2} + \lambda^{-4}, 0]$. Let $h_t$ be the flow generated by $v_{\beta_0}^s$ for some fixed $- \lambda^2 + \lambda^{-4} < \beta_0 < - \lambda^2 +1 $. In fact, if we choose for the function $k$ any $\mathcal{C}^2$ unimodal and even function supported in $[-1,1]$, equal to $1$ at $0$ and satisfying $k+id \, k' \leqslant 1$, the induced vector fields $v_\beta^s$ enjoys the same properties as before, but they are also $\mathcal{C}^1$ -- see the discussion in \cite[Theorem 3.7]{carrand2020ARPE} -- hence the flow $h_t$ is also $\mathcal{C}^1$. We make such a choice for $k$.
We claim that this flow $h_t$ satisfies the condition of Theorem \ref{thm:_main} and that it is not minimal.

In order to prove this result, we first construct a closed transversal curve $\gamma$. We then construct a particular homotopy between the first return map and a rigid rotation, where none of the in-between map has a periodic point. From the continuity of the rotation number, it is enough to compute the rotation number of the rigid rotation, which happens to be a quadratic integer. The nonminimality follows from the invariance of the proper closed set $K_{\beta_0}$ by the flow $h_t$. First we need the following lemma.

\begin{lemma}\label{no_periodic_point}
The flow $h_t$ does not have periodic orbit. This is also true for the flow generated by $v^s_{\beta}$ for any $- \lambda^2 + \lambda^{-4} < \beta \leqslant 0 $.
\end{lemma}
\begin{proof}
By construction, each vector field $v^s_{\beta}$ satisfies $\mathrm{d}_x f_{\beta} (v^s_{\beta}(x)) = \lambda^{-2} v^s_{\beta}(f_{\beta}(x))$. By differentiating $f_{\beta_0} \circ h_t(x)$ and $h_{\lambda^{-2} t} \circ f_{\beta_0}(x)$ according to $t$, we get that these two functions satisfy the same Cauchy problem for all $x \in \mathbbm{T}^2$, thus the relation  
\begin{align}\label{eq:commut_flow_map}
f_{\beta_0} \circ h_t = h_{\lambda^{-2} t} \circ f_{\beta_0}
\end{align} holds by uniqueness of the solution (because $v^s_\beta$ is Lipschitz continuous). Therefore, if by contradiction $h_t$ has a periodic orbit, by applying $f_{\beta_0}^n$, for $n$ large enough, we get an arbitrarily short periodic orbit for the flow. This contradicts the fact that the component along $e_s$ in the basis $(e_u,e_s)$ of $v^s_{\beta_0}$ is constant equal to $1$.
\end{proof}

\begin{proof}[Proof of Theorem \ref{thm:_nonminimal_flow}]
Since the map $(x,\beta) \mapsto v^s_{\beta}(x)$ is continuous on the compact set $\mathbbm{T}^2 \times [\beta_0, 0]$, the component of these vector fields in the basis $(e_u,e_s)$ along $e_u$ is uniformly bounded and along $e_s$ is equal to $1$, by definition. Therefore, there exists a vector $w$ of rational slope, say $w = \frac{1}{\sqrt{p^2 + q^2}} \begin{pmatrix}
q \\ p
\end{pmatrix}$, where $p$ and $q$ are coprimes, so that $w$ is uniformly transverse to $v^s_\beta$ for all $\beta \in [\beta_0,0]$. Define $\gamma$ to be the closed curve passing through $(0,0)$ and with slope $p/q$. By choice of $w$, the curve $\gamma$ is transverse to $v^s_{\beta}$ and so for every $\beta$ in $[\beta_0, 0]$. We can naturally parametrize $\gamma$ by $\mathbbm{S}^1$.

Let $R : \mathbbm{S}^1 \to \mathbbm{S}^1$ be the first return map to $\gamma$ of $h_t$. Notice that performing a time change on this flow does not affect the first return map $R$, but only the first return time function $u$. In order to simplify computations, renormalize the vector fields as follows
\[ w^s_{\beta} = \frac{1}{ \left\langle v^s_{\beta}, w^{\perp} \right\rangle } v^s_{\beta} \]
so that, for each $\beta$, the flow $\phi^{(\beta)}_t$  generated by $w^s_{\beta}$ has a constant first return time function $u_{\beta} \equiv \tau_{\beta}$, where $w^{\perp}$ is the unitary vector equal to $w$ rotated by an angle $\pi/2$. These first return time functions do not depend on $\beta$, in other words $\tau_{\beta} \equiv \tau$. Since $b_0 \equiv 0$, notice that $w^s_0$ is a constant vector field (equals everywhere to $e_s$), hence its first return map to $\gamma$ is a rigid translation $R_{\alpha} : x \mapsto x + \alpha $. Introduce also the notation $R^{(\beta)}$ for the first return map to $\gamma$ of $\phi^{(\beta)}_t$. In particular $R= R^{(\beta_0)}$ and $R_\alpha = R^{(0)}$.

By \cite[Theorem 3.10]{carrand2020ARPE}, the map $\beta \mapsto v^s_{\beta}$ is continuous for the $\mathcal{C}^0$-topology on the space of vector fields. From a Gronwall type argument, we get that $\beta \mapsto R^{(\beta)}$ is continuous for the $\mathcal{C}^0$-topology. Now, by \cite[Proposition II.2.7]{herman1979conjugaison}, the map $\beta \mapsto \rho(R^{(\beta)})$ is continuous, where $\rho(R^{(\beta)})$ stands for the rotation number of $R^{(\beta)}$. In order to prove that $\rho(R)=\alpha$, we prove that $\rho(R^{(\beta)})$ cannot be rational, but this directly follows from Lemma \ref{no_periodic_point}. Hence $\beta \mapsto \rho(R^{(\beta)})$ is a constant map and $\rho(R)=\alpha$.

We now compute the value of $\alpha$. Consider lifts $\tilde{w}^s_{(0)}$, $\tilde{\gamma}$ and $\tilde{\phi}^{(0)}_t $ to $\mathbbm{R}^2$ of respectively $w^s_0$, $\gamma$ and $\phi^{(0)}_t$. Let $(\partial_x,\partial_y)$ be the canonical basis of $\mathbbm{R}^2$. Notice that the arc $\{ \tilde\phi^{(0)}_t ((0,1)) \mid -p \tau  \leqslant t \leqslant 0 \} $ starts at the point $(0,1)$ and ends on the branch of $\tilde{\gamma}$ containing $(0,0)$ at some point $c w$, for some $c>0$. The coordinates of this intersection point satisfy the system of equations
\begin{empheq}[left=\empheqlbrace]{align*}
-p \tau \left\langle w^s_{(0)}, \partial_x \right\rangle &= c q (p^2 + q^2)^{-1/2} \\
1 - p \tau \left\langle w^s_{(0)}, \partial_y \right\rangle &= c p (p^2 + q^2)^{-1/2},
\end{empheq}
where $\left\langle \cdot , \cdot \right\rangle$ denotes the usual scalar product. Now, notice that $c / |\gamma| = -p \alpha$, where $|\gamma|$ is the length of the closed curve $\gamma$. We can solve these equations for $\alpha$ and get 
\[\alpha = \frac{1}{pq} \frac{1}{\lambda - \frac{p}{q}} \]
which clearly is a quadratic integer, since $\lambda$ is. Therefore $\alpha$ is of constant type.

The nonminimality of $h_t$ is ensured by properties proven in \cite[Chapter 9]{coudene2016attractor}. More precisely, let $U$ be the basin of attraction of $(0,0)$ for $f_{\beta_0}$ and $K$ be its complement in the torus. In \cite[Chapter 9]{coudene2016attractor}, Coud{\`e}ne proved that the set $K$ is nonempty and that $U$ and $K$ are invariant by $f_{\beta_0}$. Now, because of (\ref{eq:commut_flow_map}), 
the sets $U$ and $K$ are invariant by the flow $h_t$.

Finally, the map $f$ is an Axiom A diffeomorphism since $f$ is transitive \cite[Chapter 9]{coudene2016attractor} on the hyperbolic set $K$ \cite[Theorem 2.9]{carrand2020ARPE}. Therefore, by the shadowing lemma, periodic points are dense in the compact invariant set $K$ which coincides with the nonwandering set of $f$.
\end{proof}

\begin{figure}[ht]
\begin{center}
\raisebox{-0.5\height}{\includegraphics[width=0.58\textwidth]{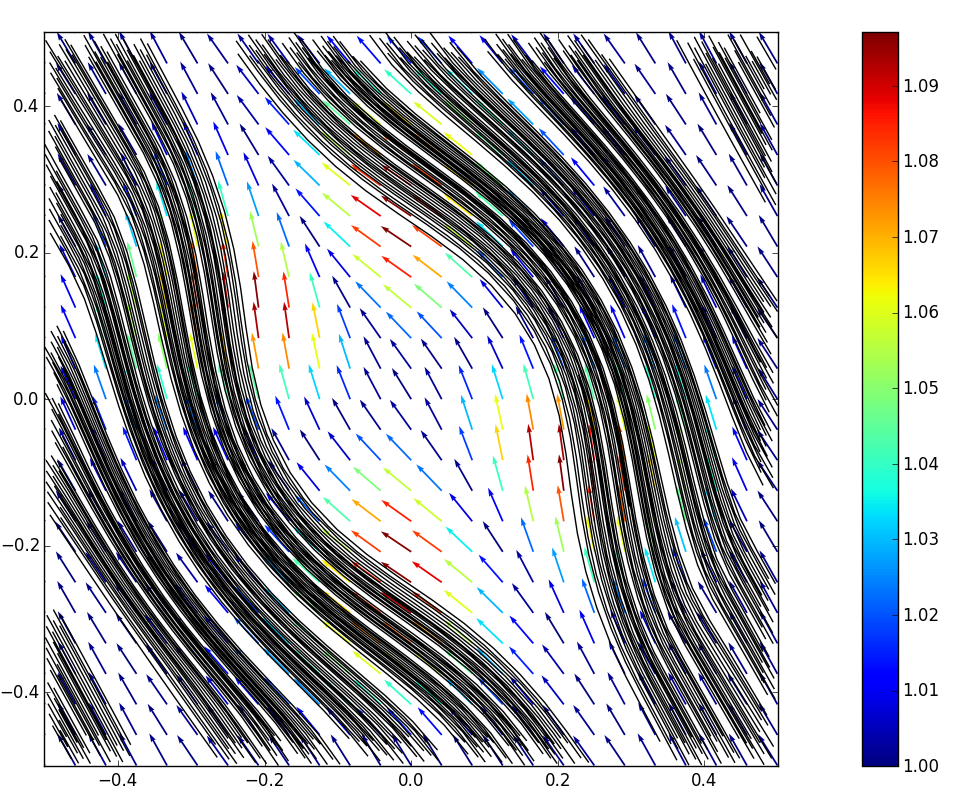}}
\end{center}
\caption{\label{fig:K_m1} Representation of the minimal component $K$ of the flow $(h_t)$. Underneath is the vector field $v^s$ generating the flow.}
\end{figure}

Finally, we give in Figure \ref{fig:K_m1} a representation of the set $K$. In \cite[Chapter 9]{coudene2016attractor}, it is proven that $K$ is the closure of the stable leaf $W^s(p)$ of a hyperbolic fixed point $p$ for $f_{\beta_0}$. From the relation (\ref{eq:commut_flow_map}) and the Hartman-Grobman theorem, it follows that this stable leaf is equal to the orbit of $p$ by the flow $h_t$. From \cite[Theorem 3.3.4]{cornfeld2012ergodic}, the set $K \cap \gamma$ coincides with any $\omega$-limit set and any $\alpha$-limit set of $R$. Therefore, the set $K$ is the minimal component of $h_t$ and is also an attractor for both positive and negative times. Moreover, $K$ is also the support of the unique invariant measure $\mu$ of $h_t$.

\appendix
\section{Alternative proof of Theorem \ref{thm:_nonminimal_flow} from semi-conjugacy}

We give an alternative proof of Theorem \ref{thm:_nonminimal_flow}. More precisely, we use the same example, but we compute the rotation number in a different way: we construct a semi-conjugacy map $h$ so that $h \circ R = R_{\alpha} \circ h$. It will follow that the rotation number of $R$ is $\alpha$. The construction of $h$ is inspired from the proof of \cite[Proposition 7]{yoccoz2005echanges}.

\begin{proof}
Exactly as in the first proof of Theorem \ref{thm:_nonminimal_flow}, we construct the closed transversal curve $\gamma$ and we renormalize the vector fields $v^s_{\beta}$ so that the time of first return function to $\gamma$ of their associated flows is constant. The computation of $\alpha$ remains the same, and we get that $\alpha$ is a quadratic integer, hence $\alpha$ is of constant type. In particular, the rotation $R_{\alpha}$ is minimal.

We now prove that the first return map $R$ of $h_t$ is semi-conjugated to $R_{\alpha}$. To this end, we construct a surjective and continuous function $h$ of the circle.

Let $h(R^n(0)) \coloneqq R_{\alpha}^n(0)$ for all $n \in \mathbbm{Z}$. This map is well defined since $h_t$ has no periodic orbit by Lemma \ref{no_periodic_point}, so does $R$. In order to extend $h$ into a continuous map, we first prove that it preserves order of triplets. Fix an orientation of $\mathbbm{S}^1$ --- and therefore of $\gamma$ --- seen as $\mathbbm{R}/\mathbbm{Z}$. Let $x_1 \coloneqq R^{n_1}(0)$, $x_2 \coloneqq R^{n_2}(0)$ and $x_3 \coloneqq R^{n_3}(0)$ be so that $(x_1,x_2,x_3)$ is an ordered triplet of $\mathbbm{S}^1$ --- we can assume that $n_1$, $n_2$ and $n_3$ are distinct. We prove that the triplet $(x_1',x_2',x_3')=(h(x_1),h(x_2),h(x_3))$ is also ordered. Consider the family of curves $\varphi_{\beta} \coloneqq \{ \phi^{(\beta)}_t(0) \mid \min(n_1,n_2,n_3)\tau \leqslant t \leqslant \max(n_1,n_2,n_3)\tau \}$. By continuity of $(x,\beta) \mapsto w^s_{\beta}(x)$, this family depends on $\beta$ in a continuous fashion. 

Notice that points $x_1$, $x_2$ and $x_3$ correspond to some intersection points between $\varphi_{\beta_0}$ and $\gamma$, and that points $x_1'$, $x_2'$, and $x_3'$ correspond to some intersection points between $\varphi_{0}$ and $\gamma$. Furthermore, we can connect $x_1$ to $x_1'$ (respectively $x_2$ to $x_2'$, and $x_3$ to $x_3'$) with intersection points between $\gamma$ and $\varphi_{\beta}$ when varying the value of $\beta$. 
Therefore we can track the evolution of $(x_1,x_2,x_3)$ with continuous functions $(x_1(\beta),x_2(\beta),x_3(\beta))$ of $\beta$ such that $x_1(\beta_0)=x_1$ and $x_1(0)=x_1'$ --- and similarly for $x_2(\beta)$ and $x_3(\beta)$.

By contradiction, suppose that the triplet $(x_1',x_2',x_3')$ is not ordered. By continuity, this means that for some value of $\beta_1$ in $[\beta_0,0]$ and without loss of generality $x_1(\beta_1)=x_2(\beta_1)$. In other words, this means that the first return map to $\gamma$ of $\phi^{(\beta_1)}_t$ has a periodic point, which contradicts Lemma \ref{no_periodic_point}.

Therefore, the map $h$ can be lifted into a ``degree'' one, increasing, function $\tilde{h} : \pi^{-1} \{R^n(0) \mid n \in \mathbbm{Z} \} \to \pi^{-1} \{R_{\alpha}^n(0) \mid n \in \mathbbm{Z} \}$, where $\pi : \mathbbm{R} \to \mathbbm{R}/\mathbbm{Z}$ is the canonical projection. In other words, $\pi \circ \tilde{h} = h \circ \pi$ and $\tilde{h}(x+1)-\tilde{h}(x) =1$ for all $x$ where $\tilde{h}$ is defined. By minimality of $R_{\alpha}$, the range of $\tilde{h}$ is dense in $\mathbbm{R}$. 
Hence, we can uniquely extend $\tilde{h}$ by a continuous, increasing and surjective function $\tilde{h} : \mathbbm{R} \to \mathbbm{R}$. Its projection on the circle, still noted $h$, is also continuous and extends $h$ into a degree one map of the circle. By continuity of $R$ and of $R_{\alpha}$, we get that $h \circ R = R_{\alpha} \circ h$. 
Therefore, by \cite[Proposition II.2.10]{herman1979conjugaison}, the rotation number of $R$ is $\alpha$, a quadratic integer.

The nonminimality of $h_t$ is ensured by properties proven in \cite[Chapter 9]{coudene2016attractor}.
\end{proof}

\begin{remark}
The construction of the conjugacy map $h$ comes from the following heuristic. Since the stable manifold of $0$ under the cat map is blown up into an open set, the basin of attraction $U_\beta \coloneqq \mathbbm{T}^2 \smallsetminus K_\beta$ of $0$ under $f_\beta$, we expect that the map $h$ relates the orbit of $0$ under $R_\alpha$ with the orbit of $I$ under $R$, where $I$ is the connected component of $\gamma \cap U_\beta$ containing $0$ (notice that $I$ is a wandering interval and that its orbit under $R$ is $\gamma \cap U_\beta$, which is dense in $\gamma$). More precisely, we expect $h$ to be similar to the Cantor staircase function, being constant when restricted to each $R^n(I)$. As in the construction of the Cantor staircase function, we only need to know the values of $h$ where it is constant, as long as $h$ is non-decreasing and that this set of values has a connected closure. In the proof above, we chose to define $h$ first by setting $h(x_n) = R^n_\alpha(0)$ with $x_n=R^n(0)$, but we could have chosen any sequence $x_n \in R^n(I)$.
\end{remark}

\bibliography{Borne_log}{}
\bibliographystyle{abbrv}
\nocite{*}

\end{document}